\documentclass[a4paper,reqno,final,11pt]{amsart}
\usepackage{amssymb}
\usepackage{amsmath}
\usepackage{amsfonts}
\usepackage{hyperref}
\usepackage{graphicx}

\newtheorem{theorem}{Theorem}
\theoremstyle{plain}
\newtheorem{corollary}{Corollary}

\newtheorem{lemma}{Lemma}

\newtheorem{remark}{Remark}
\numberwithin{equation}{section}

\title[Observability and controllability of the 1--d wave equation]{%
Observability and controllability of the 1--d wave equation in domains with
moving boundary}
\author{Abdelmouhcene SENGOUGA}
\address[Abdelmouhcene Sengouga]{ Laboratory of Functional Analysis and
Geometry of Spaces, University of M'sila, 28000 M'sila, Algeria.}
\email{amsengouga@gmail.com}
\date{\today.}
\begin{document}

\begin{abstract}
By mean of generalized Fourier series and Parseval's equality in weighted $%
L^{2}$--spaces, we derive a sharp energy estimate for the wave equation in a
bounded interval with a moving endpoint. Then, we show the observability, in
a sharp time, at each of the endpoints of the interval. The observability
constants are explicitly given. Using the Hilbert Uniqueness Method we
deduce the exact boundary controllability of the wave equation.

\end{abstract}
\maketitle
\section{Introduction}

In this work, we will consider transverse oscillations of a uniform string
whose length varies linearly with time. For $t\geq t_{0}>0$, we first denote%
\begin{equation*}
\Omega _{t}:=\left( 0,\ell t\right) ,
\end{equation*}%
which is an interval with the right endpoint depending on time. We assume
that 
\begin{equation}
0<\ell <1,  \label{tlike}
\end{equation}%
i.e. the length of $\Omega _{t}$ is increasing with a constant speed $\ell $
less then $1$. The case of a fixed interval, i.e. $\ell =0,$ and intervals
with a fast moving endpoint, $\ell \geq 1$, will not be considered here. We
prefer to follow Balazs \cite{Bal:61} and take the initial time $t_{0}>0.$ We can
take the initial time to be $t_{0}=0$ and an initial domain $\left( 0,\ell
_{0}\right) $, but this will complicate the mathematics and the Fourier
formulas obtained below.

Let $T>t_{0}\ $and consider the\ following non-cylindrical domain, and its
lateral boundary,%
\begin{equation*}
Q_{T}:=\bigcup_{t_{0}<s<T}\left\{ \Omega _{s}\times \left\{ s\right\}
\right\} ,\ \ \ \Sigma _{T}:=\bigcup_{t_{0}<s<T}\left\{ \partial \Omega
_{s}\times \left\{ s\right\} \right\} .
\end{equation*}%
Off course (\ref{tlike}) ensure the so-called time-likeness condition $%
\left\vert \nu _{t}\right\vert \leq \left\vert \nu _{x}\right\vert $ on$\
\Sigma _{t},\ $for $t>0,$ where $\nu =(\nu _{t},\nu _{x})$ is the unit
outward normal on the lateral boundary $\Sigma _{t}$. Let us now consider
the following wave equation with homogeneous Dirichlet boundary conditions 
\begin{equation}
\left\{ 
\begin{array}{ll}
\phi _{tt}-\phi _{xx}=0,\ \medskip & \text{in }Q_{T}, \\ 
\phi \left( 0,t\right) =0,\text{ \ \ }\phi \left( \ell t,t\right) =0, \medskip & \ 
\text{for \ }t\in \left( t_{0},T\right) , \\ 
\phi (x,t_{0})=\phi ^{0}\left( x\right) \text{, \ \ }\phi _{t}\left(
x,t_{0}\right) =\phi ^{1}\left( x\right) , & \ \text{for \ \ }x\in \Omega
_{t_{0}},%
\end{array}%
\right.  \label{wave}
\end{equation}%
where $\phi (x,t)$ is the transverse displacement of the string and the
subscripts $t$ and $x$ stand for the derivatives with respect to time and
space, respectively. This is an example of evolution problems in
non-cylindrical domains arising in many important applications (such as biology,
engineering, quantum mechanics,...), see the survey paper by Knobloch and Krechetnikov \cite{KnK:14}.

Under the assumption (\ref{tlike}), it is by now well known that for every
initial data%
\begin{equation*}
\phi ^{0}\in H_{0}^{1}\left( \Omega _{t_{0}}\right) ,\phi ^{1}\in
L^{2}\left( \Omega _{t_{0}}\right) \ 
\end{equation*}%
there exists a unique solution to Problem (\ref{wave}) such that%
\begin{equation*}
\phi \in C\left( [t_{0},T];H_{0}^{1}\left( \Omega _{t}\right) \right) ,\text{
\ }\phi _{t}\in C\left( [t_{0},T];L^{2}\left( \Omega _{t}\right) \right) ,
\end{equation*}%
see \cite{BaC:81,Lion:69}. We define the "energy" of the above problem as 
\begin{equation*}
E\left( t\right) =\frac{1}{2}\int_{0}^{\ell t}\phi _{x}^{2}+\phi _{t}^{2}%
\text{ }dx\ \ \ \ \ \text{for }t\geq t_{0}.
\end{equation*}%
which is not a conserved quantity in time, in contrast with the wave equation
in cylindrical domains.

Let $\xi \in \left\{ 0,\ell t\right\} $ and denote $T_{0}:=T-t_{0}.$ The
problem of observability of (\ref{wave}) at the boundary $x=\xi $ can be
formulated as follows: To give sufficient conditions on $T_{0}$ such that
there exists $C(T_{0})>0$ for which the following inequality holds for all
solutions of (\ref{wave}): 
\begin{equation}
E\left( t_{0}\right) \leq C(T_{0})\int_{t_{0}}^{t_{0}+T_{0}}\phi
_{x}^{2}\left( \xi ,t\right) +\phi _{t}^{2}\left( \xi ,t\right) dt.
\label{ct}
\end{equation}%
This is the so-called observability inequality, which allows estimating the
energy of solutions in terms of the energy localized at the boundary $x=\xi $%
. The best value of $C(T_{0})$ is the observability constant. Due to the finite speed of propagation (here
equal to $1$), the time $T_{0}$ should be sufficiently large and one expects
that it depends on the initial length of $\Omega _{t_{0}}$ and\ also on the
speed of expansion $\ell .$

On the other hand, we consider the following boundary controllability
problem: Given 
\begin{equation*}
(y^{0},y^{1})\in L^{2}\left( \Omega _{t_{0}}\right) \times H^{-1}\left(
\Omega _{t_{0}}\right) \ \text{\ and \ }(y_{T}^{0},y_{T}^{1})\in L^{2}\left(
\Omega _{t}\right) \times H^{-1}\left( \Omega _{t}\right) ,
\end{equation*}%
find a control function $v\in L^{2}\left( t_{0},T\right) ,$ acting at the
boundary $x=\xi ,$ such that the solution of%
\begin{equation}
\left\{ 
\begin{array}{ll}
y_{tt}-y_{xx}=0,\ \medskip  & \text{in }Q_{T}, \\ 
\displaystyle y\left( 0,t\right) =\left( 1-\frac{\xi }{\ell t}\right) v\left( t\right) ,%
\text{ \ \ }y\left( \ell t,t\right) =\frac{\xi }{\ell t}v\left( t\right) , \medskip & 
\ \text{for \ }t\in \left( t_{0},T\right) , \\ 
y(x,t_{0})=y^{0}\left( x\right) ,\text{ \ \ }y_{t}\left( x,t_{0}\right)
=y^{1}\left( x\right) , & \ \text{for \ \ }x\in \Omega _{t_{0}},%
\end{array}%
\right.   \label{wavec}
\end{equation}%
satisfies 
\begin{equation}
y\left( T\right) =y_{T}^{0},\quad y_{t}\left( T\right) =y_{T}^{1}.
\label{yT}
\end{equation}%
Problem (\ref{wavec}) admits a unique solution in the sense of
transposition, see \cite{Mir:96,MML:13},%
\begin{equation*}
y\in C([t_{0},T];L^{2}\left( \Omega _{t}\right) )\cap
C^{1}([t_{0},T];H^{-1}\left( \Omega _{t}\right) ).
\end{equation*}%
To show the controllability, we use the \emph{Hilbert Uniqueness Method}
(HUM) introduced in the seminal work of Lions \cite{Lion:88}, see also %
Komornik \cite{Komo:94}. The method reduces the controllability problem to the
observability of the homogeneous problem (\ref{wave}).

The controllability of the wave equation in non-cylindrical domains was
considered by several authors. Using the multiplier method, Bardos and Chen \cite{BaC:81}
derived some decay estimates for the wave equation in time-like domains,
then show the exact internal controllability of the wave equation by a
stabilization approach. Using a suitable change of variables, Miranda \cite{Mir:96}
transforms the non-cylindrical domain to a cylindrical one, obtaining a new
operator with variable coefficients, then he shows the exact boundary
controllability by HUM. Recently, there has been a renewed interest in such controllability problems. In particular, Problem (\ref{wavec}) was considered, for instance,
by \cite{CJW:15,CLG:13,SLL:15} where the controllability is established by the
multiplier method. Although this is a one-dimensional problem,  no one so far, to my knowledge, gave
the minimal time of controllability or observability and specified the constant of
observability.

Going back to problems in cylindrical domains, the first results of
observability and controllability of evolution problems was obtained by
Fourier series, see Komornik and Loreti \cite{KoLo:05}, Russell \cite{Rus:78} and the references cited therein.
This is not the case for non-cylindrical domains. To the author's knowledge,
this paper is the first attempt to apply Fourier series techniques to
establish observability and controllability results for the wave equation in
non-cylindrical domains. Further results will be presented elsewhere. 

In this work, we first express the solution of Problem (\ref{wave}) by a Fourier formula,
then an estimate of the energy is derived by using Parseval's equality in a
weighted $L^{2}$ space. In agreement with precedent works, for instance \cite%
{SLL:15}, the energy decays as $t^{-1}.$ Next, we show that Problem (\ref%
{wave}) is observable at the fixed endpoint of the interval as well as at
the moving one. The observability constants are explicitly given and
obviously depend on $\ell .$ The sharp time of observability, $T_{0}=2\ell
t_{0}/\left( 1-\ell \right) ,$ turn out to be the same for both endpoints.
Using HUM we obtain the exact boundary controllability, at one of the
endpoints, of Problem (\ref{wavec}) for $T_{0}\geq 2\ell t_{0}/\left( 1-\ell
\right) $. This improves some recent results on the controllability of (\ref%
{wavec}) obtained by the multiplier method. In particular, taking the
initial length $\ell t_{0}=1$ and assuming that (\ref{tlike}) holds,

$\bullet $ Cui et al. \cite{CJW:15} showed the controllability of (\ref{wavec}), at
the fixed endpoint, for a larger time $T_{1}>\left( \exp \left( \frac{2\ell
\left( 1+\ell \right) }{\left( 1-\ell \right) ^{3}}\right) -1\right) /\ell $.

$\bullet $ Sun et al. \cite{SLL:15} showed the controllability of (\ref{wavec}), at
the moving endpoint, for $T_{2}>2/\left( 1-\ell \right) $. However, they did
not address the limiting case $T_{2}=2/\left( 1-\ell \right) $ and its
optimality. Cui et al. \cite{CLG:13}\emph{\ }obtained a larger time of
controllability\linebreak $T_{3}>\left( \exp \left( \frac{2\ell \left(
1+\ell \right) }{\left( 1-\ell \right) }\right) -1\right) /\ell .$

The remainder of this paper is organized as follows. In section 2, we give
the exact solution of Problem (\ref{wave}) then derive a sharp estimate for
the energy. Next, the boundary observability and controllability at the
fixed endpoint and at the moving endpoint are shown in the third and fourth
section, respectively.

\section{Energy estimates}

\subsection{Exact solution}

Let $a,b\in 
\mathbb{R}
,$ such that $a>b$ and consider a nonnegative (weight) function \linebreak $w:\left(
a,b\right) \rightarrow 
\mathbb{R}
$. In the sequel, we denote by $L^{2}\left( a,b,wdx\right) $ the weighted
Hilbert space of measurable complex valued functions on $%
\mathbb{R}
,$ endowed by the scalar product%
\begin{equation*}
\int\limits_{a}^{b}f\left( x\right) \overline{g\left( x\right) }\text{ }%
w\left( x\right) dx,
\end{equation*}%
and its associated norm, see for instance Asmar \cite{Asma:05}. As usual, we drop 
$wdx$ in the space notation if $w\equiv 1$.

A well known result in analysis is that the set of functions $\left\{ \left(
1/\sqrt{2}\right) e^{in\pi z}\right\} _{n\in 
\mathbb{Z}
}$ is a complete orthonormal set in the space $L^{2}\left( 0,2\right) .$ By
making the change of variable 
\begin{equation*}
z=\alpha _{\ell }\log \left( \frac{t+x}{t(1-\ell)}\right) ,\text{ \ \ for every\ }t\geq t_{0},
\end{equation*}%
where $\alpha _{\ell }=2/\log \left( \frac{1+\ell }{1-\ell }\right)$, we
obtain the set of function $\left\{ \sqrt{\alpha _{\ell }/2}\ e^{in\pi
\alpha _{\ell }\log \left( t+x\right) }\right\} _{n\in 
\mathbb{Z}
}$ which is still a complete orthonormal set in the weighted Hilbert space
\linebreak $L^{2}\left( -\ell t,\ell t,\frac{dx}{\left( t+x\right) }\right) $%
. Note that (\ref{tlike}) ensures that the weight function $1/\left(
t+x\right) $ is positive. By consequence, every $f(\cdot ,t)\in L^{2}\left(
-\ell t,\ell t,\frac{dx}{\left( t+x\right) }\right) $ can be written as 
\begin{equation*}
f(x,t)=\sum\limits_{-\infty }^{+\infty }c_{n}e^{in\pi \alpha _{\ell }\log
\left( t+x\right) },\ \ \ \forall t\geq t_{0},
\end{equation*}%
where the coefficients $c_{n}$ are given by 
\begin{equation*}
c_{n}=\frac{\alpha _{\ell }}{2}\int\limits_{-\ell t}^{\ell t}f(x,t)e^{-in\pi
\alpha _{\ell }\log \left( t+x\right) }\frac{dx}{\left( t+x\right) },\ \ \
\forall t\geq t_{0}.
\end{equation*}%
Moreover, the following Parseval's equality holds%
\begin{equation*}
\int\limits_{-\ell t}^{\ell t}\left\vert f(x,t)\right\vert ^{2}\text{ }\frac{%
dx}{\left( t+x\right) }=\left( \frac{2}{\alpha _{\ell }}\right)
\sum\limits_{-\infty }^{+\infty }\left\vert c_{n}\right\vert ^{2},\ \ \
\forall t\geq t_{0}.
\end{equation*}

Arguing as in \cite{Bal:61}, we can show that the exact solution of Problem (%
\ref{wave}), is the restriction to the interval $\left( 0,\ell t\right) $ of
the following function$\ $given by the generalized Fourier formulas%
\begin{equation*}
\phi (x,t)=\sum\limits_{-\infty }^{+\infty }C_{n}\left( e^{in\pi \alpha
_{\ell }\log \left( t+x\right) }-e^{in\pi \alpha _{\ell }\log \left(
t-x\right) }\right) ,\text{ \ \ }x\in \left( -\ell t,\ell t\right) ,t\geq
t_{0}.
\end{equation*}%
Note that $\phi $ is an odd function of $x$ which ensure that the condition $%
\phi \left( 0,t\right) =0$ is satisfied for every $t\geq t_{0}.$ The
coefficients $C_{n}$ are\ complex numbers, independent of $t$, given by 
\begin{equation*}
C_{0}=0\text{ \ \ and \ \ }C_{n}=\frac{1}{4n\pi i}\int\limits_{-\ell
t_{0}}^{\ell t_{0}}\left( \phi _{x}^{0}+\phi ^{1}\right) e^{-in\pi \alpha
_{\ell }\log \left( t_{0}+x\right) }dx\text{, \ \ if\ }n\neq 0,
\end{equation*}%
where the initial conditions $\phi ^{0}$ and $\phi ^{1}$ are also considered
as odd functions on $x$ defined on the interval $\left( -\ell t_{0},\ell
t_{0}\right) $.

\subsection{Energy estimates}

The following lemma gives the decay rate of the energy.

\begin{lemma}
\label{th1}Under the assumption \emph{(\ref{tlike}),} the solution of
Problem \emph{(\ref{wave})} satisfies 
\begin{equation}
tE\left( t\right) +\int\limits_{0}^{\ell t}x\phi _{x}\phi _{t}\ dx=S_{\ell },%
\text{ \ \ for \ }t\geq t_{0},  \label{est0}
\end{equation}%
where $S_{\ell }:=2\pi ^{2}\alpha _{\ell }\sum\limits_{-\infty }^{+\infty
}\left\vert nC_{n}\right\vert ^{2}$ is independent of $t,$ and it holds that%
\begin{equation}
\frac{S_{\ell }}{\left( 1+\ell \right) t}\leq E\left( t\right) \leq \frac{%
S_{\ell }}{\left( 1-\ell \right) t},\ \ \ \ \ \text{\ for }t\geq t_{0}.
\label{ES}
\end{equation}
\end{lemma}

\begin{proof}
First, we deduce from the exact formula of the solution that 
\begin{eqnarray}
\phi _{t}(x,t) &=&\sum\limits_{-\infty }^{+\infty }in\pi \alpha _{\ell
}C_{n}\left( \dfrac{e^{in\pi \alpha _{\ell }\log \left( t+x\right) }}{(t+x)}-%
\dfrac{e^{in\pi \alpha _{\ell }\log \left( t-x\right) }}{(t-x)}\right) , \smallskip
\label{pht} \\
\phi _{x}(x,t) &=&\sum\limits_{-\infty }^{+\infty }in\pi \alpha _{\ell
}C_{n}\left( \dfrac{e^{in\pi \alpha _{\ell }\log \left( t+x\right) }}{(t+x)}+%
\dfrac{e^{in\pi \alpha _{\ell }\log \left( t-x\right) }}{(t-x)}\right) ,
\label{phx}
\end{eqnarray}%
which are respectively an odd and an even functions of $x,$ for every $t\geq
t_{0}$. In particular we deduce that%
\begin{equation*}
(t+x)\left( \phi _{x}+\phi _{t}\right) =2\pi \alpha _{\ell
}\sum\limits_{-\infty }^{+\infty }inC_{n}e^{in\pi \alpha _{\ell }\log \left(
t+x\right) }.
\end{equation*}%
Thanks to the Parseval's equality, applied to $(t+x)\left( \phi _{x}+\phi
_{t}\right) $ as a function in the Hilbert space $L^{2}\left( -\ell t,\ell t,%
\frac{dx}{\left( t+x\right) }\right) $, we have%
\begin{equation}
\int\limits_{-\ell t}^{\ell t}(t+x)\left( \phi _{x}+\phi _{t}\right)
^{2}dx=\int\limits_{-\ell t}^{\ell t}\left\vert (t+x)\left( \phi _{x}+\phi
_{t}\right) \right\vert ^{2}\frac{dx}{\left( t+x\right) }=\left( \frac{2}{%
\alpha _{\ell }}\right) 4\pi ^{2}\alpha _{\ell }^{2}\sum\limits_{-\infty
}^{+\infty }\left\vert nC_{n}\right\vert ^{2}  \notag
\end{equation}%
i.e. 
\begin{equation}
\int\limits_{-\ell t}^{\ell t}(t+x)\left( \phi _{x}+\phi _{t}\right)
^{2}dx=4S_{\ell }\text{ \ \ for }t\geq t_{0}.  \label{E2}
\end{equation}%
Both sides are finite since $\phi _{x},\phi _{t}\in L^{2}\left( -\ell t,\ell t\right)$. Changing $x$ by $-x$ in the last formula, we also obtain 
\begin{equation}
\int\limits_{-\ell t}^{\ell t}(t-x)\left( \phi _{x}-\phi _{t}\right)
^{2}dx=4S_{\ell },\text{ \ \ for \ }t\geq t_{0}.  \label{E3}
\end{equation}%
Summing up (\ref{E2}) and (\ref{E3}), we infer that%
\begin{equation*}
2t\int\limits_{-\ell t}^{\ell t}\phi _{x}^{2}+\phi _{t}^{2}\
dx+4\int\limits_{-\ell t}^{\ell t}x\phi _{x}\phi _{t}\ dx=8S_{\ell },\text{
\ \ for \ }t\geq t_{0}.
\end{equation*}%
Since all the functions under the integral signs are even of $x,$ then 
\begin{equation*}
\frac{t}{2}\int\limits_{0}^{\ell t}\phi _{x}^{2}+\phi _{t}^{2}\
dx+\int\limits_{0}^{\ell t}x\phi _{x}\phi _{t}\ dx=S_{\ell },\text{ \ \ for
\ }t\geq t_{0},
\end{equation*}%
and (\ref{est0}) follows. To show (\ref{ES}), we use the inequality $\pm
ab\leq a^{2}+b^{2}$ to obtain%
\begin{equation*}
\pm \int\limits_{0}^{\ell t}x\phi _{x}\phi _{t}\ dx\leq \ell t\text{ }%
E\left( t\right) \text{, \ \ \ for }t\geq t_{0}.
\end{equation*}
Taking into account (\ref{est0}), it comes that 
\begin{equation}
t\left( 1-\ell \right) E\left( t\right) \leq S_{\ell }\text{ \ \ and\ \ \ }%
t\left( 1+\ell \right) E\left( t\right) \geq S_{\ell }.  \label{ESES}
\end{equation}
This implies (\ref{ES}) and the lemma is proved.
\end{proof}

\begin{remark}
Since (\ref{ESES}) holds also for $t=t_{0}$, then we get%
\begin{equation*}
t\left( 1-\ell \right) E\left( t\right) \leq t_{0}\left( 1+\ell \right)
E\left( t_{0}\right) \text{ \ \ and\ \ \ }t\left( 1+\ell \right) E\left(
t\right) \geq t_{0}\left( 1-\ell \right) E\left( t_{0}\right) ,
\end{equation*}%
for $t\geq t_{0}.$ Thus the energy $E\left( t\right) $ satisfies%
\begin{equation*}
\frac{\left( 1-\ell \right) t_{0}E\left( t_{0}\right) }{\left( 1+\ell
\right) t}\leq E\left( t\right) \leq \frac{\left( 1+\ell \right)
t_{0}E\left( t_{0}\right) }{\left( 1-\ell \right) t},\ \ \ \ \ \text{\ for }%
t\geq t_{0}.
\end{equation*}%
An analog inequality was obtained in \cite{SLL:15}, by the multiplier method.
\end{remark}

\section{Observability and controllability at the fixed endpoint.}

In this section, we show the observability of (\ref{wave}) at $\xi =0$, then
by applying HUM we deduce the exact controllability of (\ref{wavec}). First,
we can state the following Lemma.

\begin{lemma}
\label{lmx0}Under the assumption \emph{(\ref{tlike}),} the solution of \emph{%
(\ref{wave})} satisfies 
\begin{equation}
\int_{t_{0}}^{\left( \frac{1+\ell }{1-\ell }\right) ^{M}t_{0}}t\phi
_{x}^{2}(0,t)dt=4MS_{\ell }  \label{=slf}
\end{equation}%
for any integer $M\geq 1,$ and it holds that%
\begin{equation}
4Mt_{0}\left( 1-\ell \right) E\left( t_{0}\right) \leq \int_{t_{0}}^{\left( 
\frac{1+\ell }{1-\ell }\right) ^{M}t_{0}}t\text{ }\phi _{x}^{2}(0,t)dt\leq
4Mt_{0}\left( 1+\ell \right) E\left( t_{0}\right) .  \label{EB0}
\end{equation}
\end{lemma}

\begin{proof}
Taking $x=0$ in (\ref{phx}), we obtain%
\begin{equation}
t\phi _{x}(0,t)=2\pi \alpha _{\ell }\sum\limits_{-\infty }^{+\infty
}inC_{n}\ e^{in\pi \alpha _{\ell }\log t}.
\end{equation}%
Noting that%
\begin{equation*}
e^{in\pi \alpha _{\ell }\log \left( \left( \frac{1+\ell }{1-\ell }\right)
^{M}t_{0}\right) }=e^{in\pi \alpha _{\ell }\log t_{0}}\times e^{inM\pi
\alpha _{\ell }\log \left( \frac{1+\ell }{1-\ell }\right) }=e^{in\pi \alpha
_{\ell }\log t_{0}},
\end{equation*}%
then, we can check that $\left\{ \sqrt{\alpha _{\ell }/2M}\ e^{in\pi \alpha
_{\ell }\log t}\right\} _{n\in 
\mathbb{Z}
}$ is a complete orthonormal set in the space $L^{2}\left( t_{0},\left( 
\frac{1+\ell }{1-\ell }\right) ^{M}t_{0},\frac{dt}{t}\right) $ and by
Parseval's equality it follows that 
\begin{equation*}
\int_{t_{0}}^{\left( \frac{1+\ell }{1-\ell }\right) ^{M}t_{0}}t\phi
_{x}^{2}(0,t)dt=\int_{t_{0}}^{\left( \frac{1+\ell }{1-\ell }\right)
^{M}t_{0}}t^{2}\phi _{x}^{2}(0,t)\frac{dt}{t}=\left( \frac{2M}{\alpha _{\ell
}}\right) 4\pi ^{2}\alpha _{\ell }^{2}\sum\limits_{-\infty }^{+\infty
}\left\vert nC_{n}\right\vert ^{2}
\end{equation*}%
which implies (\ref{=slf}). The estimate (\ref{EB0}) follows by using (\ref%
{ESES}) for $t=t_{0}$.
\end{proof}

\begin{remark}
\label{RMKD0}From (\ref{EB0}) we infer that%
\begin{equation*}
\int_{t_{0}}^{\left( \frac{1+\ell }{1-\ell }\right) ^{M}t_{0}}\phi
_{x}^{2}(0,t)dt\leq 4M\left( 1-\ell \right) E\left( t_{0}\right) .
\end{equation*}%
Then, for any $T\geq t_{0},$ we can always choose an integer $M$ such that $%
T\leq \left( \frac{1+\ell }{1-\ell }\right) ^{M}t_{0}$ and since the
integrated function is nonnegative, we deduce the following (so-called
direct) inequality 
\begin{equation}
\int_{t_{0}}^{T}\phi _{x}^{2}(0,t)dt\leq 4M\left( 1-\ell \right) E\left(
t_{0}\right) .  \label{D0}
\end{equation}
\end{remark}

Noting that $\phi _{t}\left( 0,t\right) =0,\forall t\geq t_{0}$, inequality (%
\ref{ct}) reads 
\begin{equation*}
E\left( t_{0}\right) \leq C(T_{0})\int_{t_{0}}^{t_{0}+T_{0}}\phi
_{x}^{2}\left( 0,t\right) dt,
\end{equation*}%
hence Lemma \ref{lmx0} is very useful to establish the observability of (\ref%
{wave})\emph{\ }at $\xi=0$.

\begin{theorem}
Under the assumption \emph{(\ref{tlike}),} if $T_{0}\geq 2\ell t_{0}/\left(
1-\ell \right) ,$ Problem \emph{(\ref{wave})} is observable at the fixed
endpoint $\xi=0$ and it holds that 
\begin{equation}
E\left( t_{0}\right) \leq \frac{1+\ell }{4\left( 1-\ell \right) ^{2}}%
\int_{t_{0}}^{t_{0}+T_{0}}\phi _{x}^{2}(0,t)dt.  \label{obs0}
\end{equation}

Conversely, if $T_{0}<2\ell t_{0}/\left( 1-\ell \right) ,$ \emph{(\ref{wave})%
} is not observable at $\xi=0$.
\end{theorem}

\begin{proof}
Noting that $\frac{\left( 1+\ell \right) t_{0}}{\left( 1-\ell \right) }%
=t_{0}+\frac{2\ell t_{0}}{\left( 1-\ell \right) }$ and taking $M=1$ in (\ref{EB0}),
then we get 
\begin{equation*}
4t_{0}\left( 1-\ell \right) \text{ }E\left( t_{0}\right) \leq \frac{1+\ell }{%
1-\ell }t_{0}\int_{t_{0}}^{t_{0}+\frac{2}{1-\ell }\ell t_{0}}\phi
_{x}^{2}(0,t)dt
\end{equation*}%
and thus inequality (\ref{obs0}) holds for $T_{0}=2\ell t_{0}/\left( 1-\ell
\right) $ and therefore for any  \linebreak $T_{0}>2\ell t_{0}/\left( 1-\ell \right) $
as well.

To show that the observability does not hold for $T_{0}<2\ell t_{0}/\left(
1-\ell \right) $, we adapt a proof found in \cite{Zua:05} where the wave
equation is considered in a fixed interval. Set $T_{\delta }=\frac{\left(
1+\ell \right) t_{0}-2\delta }{  1-\ell } $ for some $%
\delta >0$ sufficiently small, then solve 
\begin{equation}
\left\{ 
\begin{array}{ll}
u_{tt}-u_{xx}=0,\ \medskip & \text{in }Q_{T}, \\ 
u\left( 0,t\right) =0,\text{ \ }u\left( \ell \left( t\right) ,t\right) =0, \medskip & 
\ \text{for \ }t\in \left( t_{0},T_{\delta }\right) , \\ 
u(x,s)=u^{0}\left( x\right) ,\text{ \ \ }u_{t}\left( x,s\right) =u^{1}\left(
x\right) , & \ \text{for \ \ }x\in \Omega _{s},%
\end{array}%
\right.  \label{waveu}
\end{equation}%
with data at time $s=\left( t_{0}-\delta \right) /\left( 1-\ell \right) $
with support in the subinterval $(\ell s-\delta ,\ell s),$ (see Figure 1a).
Let us check that the solution of (\ref{waveu}) is unique. In one hand, $u$
is unique for $t\geq s$ as it satisfy a wave equation, in an interval
expanding at a speed $\ell <1$, with initial data at $t=s$. On the other
hand, setting%
\begin{equation*}
U\left( t\right) =u\left( s-t\right) ,\text{ \ \ for }t\in \left(
t_{0},s\right) ,
\end{equation*}%
then $U$ is unique as it satisfy a wave equation, in an interval
contracting at a speed $\ell <1$, with initial data (see \cite{BaC:81})%
\begin{equation*}
U(x,0)=u^{0}\left( x\right) ,\ \ U_{t}\left( x,0\right)
=u^{1}\left( x\right) \text{ \ \ \ in }\left( 0,\ell s\right) .
\end{equation*}%
Thus $u$ is uniquely determined for $t\geq t_{0}.$ This solution is such
that $u_{x}(0,t)=0$ for $t_{0}+\delta <t<T_{\delta }-\delta $ since the
segment $x=0,t\in (t_{0}+\delta ,T_{\delta }-\delta )$ remains outside the
domain of influence of the segment $t=s,x\in (\ell s-\delta ,\ell s)$. This
means that%
\begin{equation*}
\int_{t_{0}+\delta }^{T_{\delta }-\delta }\phi _{x}^{2}(0,t)dt=0,\text{ \ \
\ \ }\forall \delta >0,
\end{equation*}%
and the observability inequality (\ref{obs0}) does not hold.
\end{proof}

\begin{figure}[htbp]
\centering
\includegraphics[width=120mm,height=55mm]{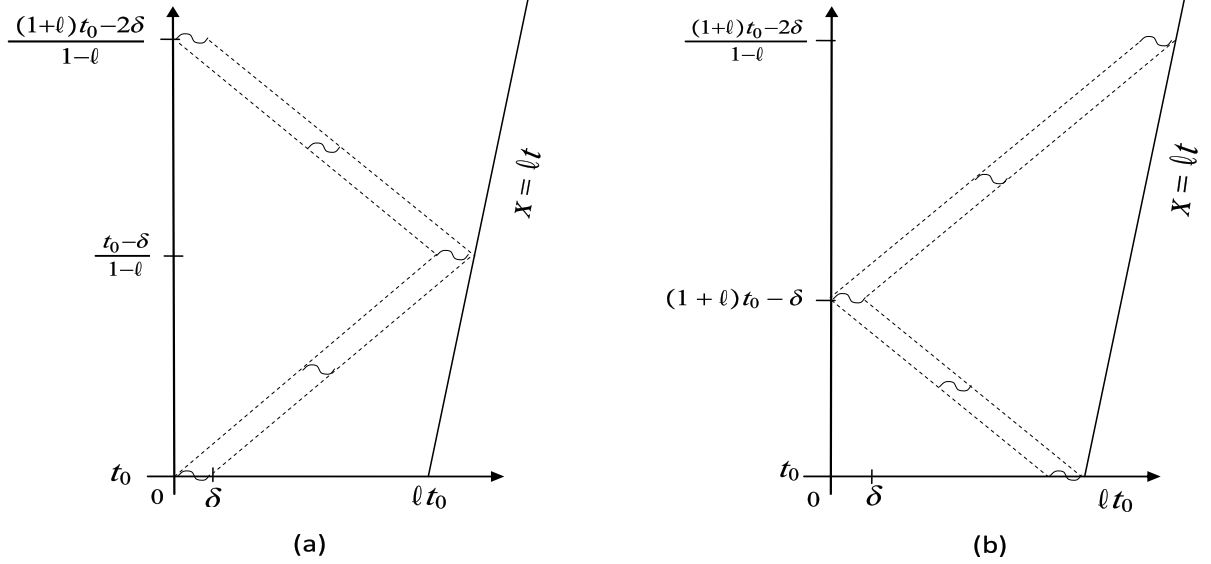}
\caption{Propagation of a wave with a small support near an endpoint}
\label{fig:figseng}
\end{figure}

\begin{remark}
\label{rmktobs}An initial disturbance concentrated near $x=0$ may propagate
to the right, as $t$ increases, and bounce back on the moving boundary$,$
when $t$ is close to $\frac{t_{0}}{1-\ell },$\ then travel to the left to
reach the fixed boundary, when $t$ is close to $\frac{1+\ell }{1-\ell }t_{0}$%
, (see Figure 1a). Thus the needed time to complete this journey is close to 
$\frac{ \left( 1+\ell \right) t_{0}}{ 1-\ell }-t_{0}=2\ell t_{0}/\left( 1-\ell \right) ,$ which is the critical time of
observability.
\end{remark}

\begin{remark}
Let us fix the initial length $\ell t_{0}=1$ (by choosing $t_{0}=1/\ell $).
Then, as $\ell \rightarrow 0,$ we recover the critical time of observability 
$T_{0}=2$ of the wave equation in the fixed interval $\left( 0,1\right) .$
\end{remark}

The idea of HUM is based on the equivalence between the observability at $%
\xi=0 $ of the homogeneous problem (\ref{wave}) and the exact controllability
at $\xi=0$ of the non-homogeneous problem (\ref{wavec}). The proof of this
equivalence for the wave equation in an interval with fixed ends (see, for
instance, pages 53--57 of \cite{Komo:94}) can be carried out without much
difficulty to yield the same result for an interval with moving ends. Whence
we have the following controllability result.

\begin{corollary}
Under the assumption \emph{(\ref{tlike}),} Problem \emph{(\ref{wavec}) }is
exactly controllable at the fixed endpoint $\xi =0$ for $T_{0}\geq 2\ell
t_{0}/\left( 1-\ell \right) .$ Moreover, we can choose the control $v$
satisfying 
\begin{equation}
\int_{t_{0}}^{t_{0}+T_{0}}v^{2}\left( t\right) dt\leq K\left( T_{0},\ell \right)
E\left( t_{0}\right) ,  \label{v0}
\end{equation}%
where $K\left( T_{0},\ell \right) $ is a constant depending on $T_{0}$ and $%
\ell .$

Conversely, if $T_{0}<2\ell t_{0}/\left( 1-\ell \right) ,$ \emph{(\ref{wavec}%
)} is not controllable at $\xi =0.$
\end{corollary}

\begin{remark}
The control obtained by HUM is $v=\phi _{x}\left( 0,t\right) $, where $\phi $
is the solution of (\ref{wave}) with some suitable choice of the initial
conditions, see also the proof of Theorem 2.1 in \cite{CJW:15}. Inequality (\ref{v0}) is a consequence of (\ref{D0}). 
\end{remark}

\section{Observability and controllability at the moving endpoint.}

In this section, we show the observability and the controllability at $\xi
=\ell t$. Let us start with the following lemma.

\begin{lemma}
\label{lmlt}Under the assumption \emph{(\ref{tlike}),} the solution of \emph{%
(\ref{wave})} satisfies 
\begin{equation}
\int_{t_{0}}^{\left( \frac{1+\ell }{1-\ell }\right) ^{M}t_{0}}t\phi
_{x}^{2}(\ell t,t)dt=\frac{4M}{\left( 1-\ell ^{2}\right) ^{2}}S_{\ell }
\label{=slm}
\end{equation}%
for any integer $M\geq 1,$ and it holds that%
\begin{equation}
\frac{4Mt_{0}}{\left( 1+\ell \right) ^{2}\left( 1-\ell \right) }E\left(
t_{0}\right) \leq \int_{t_{0}}^{\left( \frac{1+\ell }{1-\ell }\right)
^{M}t_{0}}t\phi _{x}^{2}(\ell t,t)dt\leq \frac{4Mt_{0}}{\left( 1-\ell
\right) ^{2}\left( 1+\ell \right) }\ E\left( t_{0}\right) .  \label{EBlt}
\end{equation}
\end{lemma}

\begin{proof}
Taking $x=\ell t$ in (\ref{phx}), we get%
\begin{equation*}
\phi _{x}(\ell t,t)=\pi \alpha _{\ell }\sum\limits_{-\infty }^{+\infty
}inC_{n}\left( \dfrac{e^{in\pi \alpha _{\ell }\log \left( 1+\ell \right) }}{%
\left( 1+\ell \right) }+\dfrac{e^{in\pi \alpha _{\ell }\log \left( 1-\ell
\right) }}{\left( 1-\ell \right) }\right) \frac{e^{in\pi \alpha _{\ell }\log
t}}{t},
\end{equation*}%
i.e.%
\begin{equation*}
t\phi _{x}(\ell t,t)=\frac{2\pi \alpha _{\ell }}{\left( 1-\ell ^{2}\right) }%
\sum\limits_{-\infty }^{+\infty }\left( inC_{n}\ e^{in\pi \alpha _{\ell
}\log \left( 1+\ell \right) }\right) e^{in\pi \alpha _{\ell }\log t}.
\end{equation*}%
Then, by Parseval's equality in $L^{2}\left( t_{0},\left( \frac{1+\ell }{%
1-\ell }\right) ^{M}t_{0},\frac{dt}{t}\right) $ applied to the function $%
t\phi _{x}(\ell t,t)$, we infer that%
\begin{equation*}
\int_{t_{0}}^{\left( \frac{1+\ell }{1-\ell }\right) ^{M}t_{0}}t\phi
_{x}^{2}(\ell t,t)dt=\left( \frac{2M}{\alpha _{\ell }}\right) \frac{4\pi
^{2}\alpha _{\ell }^{2}}{\left( 1-\ell ^{2}\right) ^{2}}\sum\limits_{-\infty
}^{+\infty }\left\vert nC_{n}\right\vert ^{2}
\end{equation*}%
which implies (\ref{=slm}). The estimate (\ref{EBlt}) follows by using (\ref%
{ESES}) for $t=t_{0}$.
\end{proof}

\begin{remark}
Inequality (\ref{EBlt}) also yields%
\begin{equation*}
\int_{t_{0}}^{\left( \frac{1+\ell }{1-\ell }\right) ^{M}t_{0}}\phi
_{x}^{2}(\ell t,t)dt\leq \frac{4M}{\left( 1-\ell \right) ^{2}\left( 1+\ell
\right) }\ E\left( t_{0}\right) .
\end{equation*}%
For any $T\geq t_{0},$ taking $M$ such that $T\leq \left( \frac{1+\ell }{%
1-\ell }\right) ^{M}t_{0},$ we have the following direct inequality 
\begin{equation}
\int_{t_{0}}^{T}\phi _{x}^{2}(\ell t,t)dt\leq \frac{4M}{\left( 1-\ell
\right) ^{2}\left( 1+\ell \right) }\ E\left( t_{0}\right) .  \label{Dlt}
\end{equation}
\end{remark}

Since $\phi \left( \ell t,t\right) =0,\forall t\geq t_{0},$ then $\phi
_{t}\left( \ell t,t\right) +\ell \phi _{x}\left( \ell t,t\right) =0,\forall
t\geq t_{0}.$ In this case, inequality (\ref{ct}) reads 
\begin{equation*}
E\left( t_{0}\right) \leq C^{\prime }(T_{0})\left( 1+\ell ^{2}\right)
\int_{t_{0}}^{t_{0}+T_{0}}\phi _{x}^{2}\left( \ell t,t\right) dt.
\end{equation*}
Thus Lemma \ref{lmlt} can be used to show the observability at $\xi =\ell t$.

\begin{theorem}
Under the assumption \emph{(\ref{tlike}),} if $T_{0}\geq 2\ell t_{0}/\left(
1-\ell \right) ,$ Problem \emph{(\ref{wave})} is observable at the moving
endpoint $\xi =\ell t$ and it holds that 
\begin{equation}
E\left( t_{0}\right) \leq \frac{\left( 1+\ell \right) ^{3}}{4}%
\int_{t_{0}}^{t_{0}+T_{0}}\phi _{x}^{2}(\ell t,t)dt.  \label{obslt}
\end{equation}

Conversely, if $T_{0}<2\ell t_{0}/\left( 1-\ell \right) ,$ \emph{(\ref{wave})%
} is not observable at $\xi=\ell t$.
\end{theorem}

\begin{proof}
Setting $M=1$ in (\ref{EBlt}), we infer%
\begin{equation*}
\frac{4t_{0}}{\left( 1+\ell \right) ^{2}\left( 1-\ell \right) }E\left(
t_{0}\right) \leq \frac{1+\ell }{1-\ell }t_{0}\int_{t_{0}}^{\frac{1+\ell }{%
1-\ell }t_{0}}\phi _{x}^{2}(\ell t,t)dt
\end{equation*}%
and (\ref{obslt}) holds when $T_{0}\geq 2\ell t_{0}/\left( 1-\ell \right) $.

To show that (\ref{obslt}) does not hold for $T_{0}<2\ell t_{0}/\left(
1-\ell \right) ,$ we argue as above. Consider again $T_{\delta }=\frac{\left(
1+\ell \right) t_{0}-2\delta }{  1-\ell } $  for
some $\delta >0$ sufficiently small. Solve Problem (\ref{waveu}) with data
at time $s=\left( 1+\ell \right) t_{0}-\delta $ with support in the
subinterval $(0,\delta ),$ (see Figure 1b). This solution is such that $%
u_{x}(\ell t,t)=0$ for $t_{0}+\delta <t<T_{\delta }-\delta $ since the
segment $x=\ell t,t\in (t_{0}+\delta ,T_{\delta }-\delta )$ remains outside
the domain of influence of the space segment $t=s,x\in (0,\delta ),$ hence%
\begin{equation*}
\int_{t_{0}+\delta }^{T_{\delta }-\delta }\phi _{x}^{2}(\ell t,t)dt=0,\text{
\ \ \ \ }\forall \delta >0.
\end{equation*}%
This ends the proof.
\end{proof}

\begin{remark}
Although the observability constant is different in (\ref{obs0}) and (\ref%
{obslt}), but the observability time is the same $T_{0}=2\ell t_{0}/\left(
1-\ell \right) .$ As in Remark \ref{rmktobs}, this value can be justified by
using characteristics,\ (see Figure 1b).
\end{remark}

Arguing as in the precedent section, we have the following controllability
result.

\begin{corollary}
Under the assumption \emph{(\ref{tlike}),} Problem \emph{(\ref{wavec}) }is
exactly controllable at the moving endpoint $\xi =\ell t$ for $T_{0}\geq
2\ell t_{0}/\left( 1-\ell \right) .$ Moreover, we can choose the control $v$
satisfying 
\begin{equation}
\int_{t_{0}}^{t_{0}+T_{0}}v\left( t\right) dt\leq K^{\prime }\left( \ell
,T_{0}\right) E\left( t_{0}\right) ,  \label{vlt}
\end{equation}%
where $K^{\prime }\left( T_{0},\ell \right) $ is a constant depending on $%
\ell $ and $T_{0}.$

Conversely, if $T<2\ell t_{0}/\left( 1-\ell \right) ,$ \emph{(\ref{wavec})}
is not controllable at $\xi =\ell t$.
\end{corollary}

\begin{remark}
The control obtained by HUM is $v=\phi _{x}\left( \ell t,t\right) $, with
some suitable choice of the initial conditions in (\ref{wave}), see also the proof of Theorem 1.2 in \cite{SLL:15}. Inequality (%
\ref{vlt}) is a consequence of (\ref{Dlt}).
\end{remark}


\end{document}